\documentclass[12pt,leqno,a4paper]{article}
\usepackage[T1]{fontenc}
\usepackage{color}

\usepackage[pdfborder={0 0 0}]{hyperref}

\usepackage{amssymb}
\usepackage{amsmath}
\usepackage{amsthm}
\usepackage{mathrsfs}

\usepackage[shortlabels]{enumitem}

\newtheorem{theorem}{Theorem}
\numberwithin{theorem}{section}
\newtheorem{proposition}[theorem]{Proposition}
\newtheorem{lemma}[theorem]{Lemma}
\newtheorem{corollary}[theorem]{Corollary}

\theoremstyle{remark}

\theoremstyle{definition}
\newtheorem{remark}[theorem]{Remark}

\numberwithin{equation}{section}

\newcommand\set[1]{\left\{\,#1\,\right\}}		

\DeclareMathOperator{\id}{id}					
\DeclareMathOperator{\divv}{div}				

\def\R{\mathbb{R}}

\usepackage[numbers]{natbib}

\begin{document}
\title{Estimating the convex relaxation of the ideal magnetohydrodynamics equations}
\author{Borb\'ala Fazekas \quad J\'ozsef J. Kolumb\'an}
\date{}
\maketitle

\begin{abstract}
We investigate the explicit convex relaxation of the ideal magnetohydrodynamics equations. We provide a non-trivial lower estimate on the lamination hull and an upper estimate on the $\Lambda$-convex hull, the latter providing inequalities which will be satisfied by weak limits of weak solutions of the ideal MHD equations, which serve as a model of averaged turbulent magnetohydrodynamical flows.
\end{abstract}

\section{Introduction}

The ideal magnetohydrodynamics equations (MHD for short) consist of the incompressible Euler equations with Lorentz force induced by a magnetic field, coupled with a Faraday-Maxwell induction equation and Ohm's law, to describe the evolution of the magnetic field. The system is used as a model for electrically conducting fluids, such as magnetic behaviour in plasmas and liquid metals. 
The theory of magnetohydrodynamics, including both its ideal and non-ideal (where viscosity and magnetic resistivity are allowed to be non-zero) formulations, originates from the pioneering work of Alfv\'en in \cite{Alf}. 
We further refer the reader to \cite{GBL,ST} for a more in-depth discussion on the mathematical modelling of magnetohydrodynamics.

In recent works \cite{FLSz,FLSz2,FLSz3} Faraco, Lindberg and Sz\'ekelyhidi Jr.  investigated the possibility of modelling magnetohydrodynamical turbulence arising in the ideal MHD equations by means of convex integration, generalizing the pioneering results of De Lellis and Sz\'ekelyhidi Jr. from \cite{DeL-Sz-Annals,DeL-Sz-Adm} regarding the hydrodynamic incompressible Euler equations. In particular, relying on the Tartar framework of convex integration based on compensated compactness (see e.g. \cite{Tartar}), one seeks to establish a relaxation of the system by determining the differential inclusion satisfied by the weak limits of weak solutions of the original system. Then, a typical convex integration result states that if this relaxation has a solution (often called "subsolution"), the original problem has infinitely many weak solutions which weakly converge to said subsolution. 

As suggested by Lax (\cite{Lax}), taking weak limits can be a deterministic way of substituting ensemble averaging, and hence a way to model turbulent flows when one cannot expect to deduce any meaningful information from examining the quantities involved (like velocity or pressure) point-wise. Indeed, there have been many recent works where the approach of De Lellis-Sz\'ekelyhidi Jr. was successfully generalized to other hydrodynamical models which are known to be turbulent. Often this approach has led to gaining quantitative information about the mean flow (i.e. subsolution), which also applies to all turbulent flows converging weakly to said mean flow, such as for instance the growth rate or the evolution of the interface of the turbulent mixing zone (see for instance \cite{GKEE,Mengual_Sz_vortex_sheet,Sz-KH} for the Kelvin-Helmholtz instability, \cite{GKSz,GK,GKRTE} for the Rayleigh-Taylor instability or \cite{Castro_Cordoba_Faraco,Castro_Faraco_Mengual,Castro_Faraco_Mengual_2, Foerster-Sz,Mengual,Noisette-Sz,Sz-Muskat} for the Muskat problem in the case of the incompressible porous media equations - IPM for short).

Although in \cite{FLSz,FLSz3} Faraco, Lindberg and Sz\'ekelyhidi Jr. established the existence of infinitely many weak solutions to the MHD equations via convex integration, they did not compute the explicit relaxation of the system, only showed that it is sufficiently large in an appropriate sense. Obtaining the complete explicit relaxation remains a difficult but important open problem in terms of deterministically modelling magnetohydrodynamical turbulence. The importance and interest lie specifically in the quantitative information regarding mean turbulent flows mentioned in the previous paragraph, which can usually only be inferred by knowing the precise convex-geometric structure of the relaxation. In particular, a famous instability in the context of MHD is the Velikhov (or electrothermal) instability \cite{v2,v3,v1}, which is of great interest in the development of magnetohydrodynamical propulsion technologies. There is hope that a full characterization of the convex relaxation could eventually yield similar calibre results for the Velikhov instability as those mentioned in the previous paragraph for classical hydrodynamical instabilities.

The aim of this paper is to make some advancements towards obtaining the precise structure of the convex relaxation, by giving non-trivial upper and lower estimates on it. 

The ideal MHD equations can be written as the following system of partial differential equations:
\begin{align}\label{eq:MHD}
\begin{split}
\partial_t u +\divv(u\otimes u-B\otimes B)+\nabla p =0,\\
\partial_t B +\divv(B\otimes u-u\otimes B)=0,\\
\divv u=\divv B=0,
\end{split}
\end{align}
for either $(x,t)\in\R^3 \times (0,T)$, or $(x,t)\in\mathbb T^3 \times (0,T)$
with appropriate zero-mean conditions for the fluid velocity $u(x,t)\in\R^3$ and the magnetic field $B(x,t)\in\R^3$. The fluid pressure is given by $p(x,t)\in\R$, as in the case of the usual hydrodynamic incompressible Euler equations. 
For further details, as well as the precise definition of distributional weak solutions to \eqref{eq:MHD}, we refer the reader to \cite{FLSz}, for the sake of brevity.

Using the Els\"asser variables $\alpha:=u+B$, $\beta:=u-B$, the above system can be written equivalently as
\begin{align}\label{eq:MHDE}
\begin{split}
\partial_t \alpha +\divv(\alpha\otimes \beta)+\nabla p =0,\\
\partial_t \beta +\divv(\beta\otimes \alpha)+\nabla p =0,\\
\divv \alpha=\divv \beta=0.
\end{split}
\end{align}
We aim to study the convex relaxation of \eqref{eq:MHDE} as a differential inclusion within the Tartar framework. To that end, one may rewrite \eqref{eq:MHDE} as the linear system
\begin{align}\label{eq:linear_system}
\begin{split}
\partial_t \alpha +\divv(M)=0,\\
\partial_t \beta +\divv(M^T) =0,\\
\divv \alpha=\divv \beta=0,
\end{split}
\end{align}
together with the nonlinear constraint $M=\alpha\otimes \beta+p\id$.

Since within the Tartar framework one is usually interested in generating infinitely many weak solutions to \eqref{eq:MHDE}, which one would like to be in the class $L^\infty$, we introduce a bound. Furthermore, the pressure will never be changed during the convex integration/convex relaxation scheme, so one can consider $p$ to be a given continuous function.

Consider $r>0$, $s>0$, and suppose that the pressure  also satisfies $|p|\leq rs$.
Let $z=(\alpha,\beta,M)\in Z:=\R^3\times\R^3\times\R^{3\times3}$ and define the set
\begin{align}\label{eq:nonlinear_constraints}
K_{r,s}:=\left\{z\in Z:\ M=\alpha\otimes \beta+p\id,\ |\alpha|=r,\ |\beta|=s \right\}.
\end{align}
Note that implicitly $K_{r,s}$ also depends on $p=p(x,t)$, however $p$ is considered to be a fixed given function during the convex relaxation (as mentioned above), and the aim of this paper is not to carry out a convex integration scheme, but to investigate the convex relaxation, so for simplicity we omit such a dependence from the notations. 
An adapted framework of convex integration which also treats the $(x,t)$ dependence of the set of nonlinear constraints can be found for instance in \cite{Crippa}.

The relaxation of  $K_{r,s}$ can be defined as the smallest set $\tilde K_{r,s}$ for which there holds that for all solutions of \eqref{eq:linear_system} with values in $K_{r,s}$, their weak limits take values in $\tilde K_{r,s}$. Hence we introduce 
 the associated wave cone, defined as
\begin{equation}\label{eq:wave_cone}
\Lambda=\set{\bar{z}\in Z:\ \exists\ (\xi,c)\in\mathbb{R}^3\times\mathbb R:\ \begin{pmatrix}
\bar{M} & \bar{\alpha}\\
\bar{M^T} & \bar{\beta}\\
\bar{\alpha}^T & 0\\
\bar{\beta}^T & 0
\end{pmatrix} \begin{pmatrix}
\xi \\ c
\end{pmatrix}=0,\ \xi\neq 0}.
\end{equation}
The wave cone essentially gives directions in which one is allowed to oscillate while still solving the PDE, more precisely, for any $\bar z\in\Lambda$, $h\in C^\infty(\mathbb R)$, setting $z(x,t)=\bar{z}h((x,t)\cdot(\xi,c))$ yields a solution of \eqref{eq:linear_system}. Here we also note that the condition $\xi\neq 0$ serves to eliminate trivial oscillations in time only, and is a condition that is usually imposed for such problems in the Tartar framework.

Although as already mentioned, the aim of this paper is not to carry out a convex integration scheme, only to investigate the convex relaxation itself, it is to be noted that constructing localized plane wave-like solutions for the MHD equations is more difficult, and pertains to the theory of Faraday two-forms. For further details, we once more refer the reader to \cite{FLSz}, Sections 4 and 5.

Since, as mentioned above, we are looking for a relaxation $\tilde K_{r,s}$ which contains the values of weak limits of solutions taking values in $K_{r,s}$,
in view of the relationship between convexity and weak convergence, one may consider the ($\Lambda$-)convex relaxation of $K_{r,s}$ by taking $\tilde K_{r,s}=K_{r,s}^\Lambda$, i.e. the $\Lambda$-convex hull of $K_{r,s}$. We recall that a function $f:Z\to\R$ is $\Lambda$-convex if $t\mapsto f(z+t\bar z)$ is convex for any $z\in Z$, $\bar z\in\Lambda$. The 
$\Lambda$-convex hull of a compact set $C\subset Z$ is then defined as the set of points from $Z$ which cannot be separated from $C$ by a  $\Lambda$-convex function. Equivalently, the hull is the smallest $\Lambda$-convex set containing $C$, where similarly a set is said to be $\Lambda$-convex if, for any two points $z_1,z_2$ in the set for which $z_1-z_2\in\Lambda$, the segment $[z_1,z_2]=\set{z\in Z:\ z=z_1+t(z_2-z_1),\ t\in[0,1]}$ is also contained in the set.

One can expect (see in particular the first paragraph of Section \ref{sec:lam}) that as long as the $\Lambda$-convex hull is non-trivial, it is indeed the appropriate relaxation, i.e. weak limits of weak solutions to the original problem will take values in the $\Lambda$-convex hull. The difficulty then lies in computing the hull explicitly. 

Typically, there are two approaches to achieve this: either to estimate from above by trying to find the smallest $\Lambda$-convex set containing $K_{r,s}$; or from below by computing so-called laminates, i.e. connecting points from $K_{r,s}$ with $\Lambda$-segments. We note that a priori it is possible that the upper and lower estimates cannot be improved to the point where they coincide, for instance this was the case for the stationary IPM equations in \cite{Hitruhin-Lindberg}. Nonetheless, it is also of mathematical interest to know such behaviour.

We present our main results for the two estimates in Propositions \ref{prop:lam} and \ref{prop:hull} below. While there remains a gap between these two results (see also the discussion in Section \ref{sec:conc}), to the authors' best knowledge, currently there exist no better results in either of these two  directions of trying to explicitly obtain the convex relaxation of the ideal MHD equations.

In particular, we can compare our results to those obtained in \cite{FLSz}, where 
the results regarding laminates are contained in Lemmas 6.10-6.14, and where the size of the relaxation in both the fluid and magnetic parts has to be small (controlled by certain constants $\varepsilon_\tau$, not explicit quantities involving $z$). In our Els\"asser formulation this essentially amounts to the relaxation of $M_0(z)=0$, for which already in our lower estimate of the lamination hull from Proposition \ref{prop:lam} we provide explicit inequalities (such as \eqref{eq:nnlam1pt5}), which do not require the introduction of any small error parameters. Hence, it yields an advancement towards the explicit characterization of the lamination hull.

Furthermore, providing a non-trivial upper estimate of the $\Lambda$-convex hull via the set $\overline U$ from Proposition \ref{prop:hull} is meaningful because weak limits of weak solutions of ideal MHD will still be in this set, i.e. averaged turbulent flows will satisfy the corresponding (non-strict) inequalities involved in \eqref{eq:hull}.

Finally, we mention here the result \cite{dyn} regarding the convex relaxation of the kinematic dynamo equations, which can be considered as a toy model for ideal MHD. In said paper the relaxation was already obtained from the set of first-order laminates. However, in our case there is much less freedom in certain choices due to the coupling of the Cauchy momentum equation with the induction equation, so we are unable to obtain the complete characterization of the $\Lambda$-convex hull with similar arguments as in \cite{dyn}.

We hope that the analysis presented in this paper inspires and helps the research community to further bridge the gap
 between existing results (such as \cite{FLSz}) and obtaining the complete $\Lambda$-convex hull associated with the ideal magnetohydrodynamics equations, which would essentially amount to closing the gap between Propositions \ref{prop:lam} and \ref{prop:hull}. 

\section{Main Results}

Our main results consist of a lower and upper estimate on the hull $K_{r,s}^\Lambda$, given by the two propositions in this section.

\subsection{A lower estimate on the $\Lambda$-convex hull via laminates}\label{sec:lam}

One approach to calculate the $\Lambda$-convex hull is via the so-called laminates method, i.e. define recursively the sets
\begin{align*}
K_{r,s}^{\Lambda,0}:=K_{r,s},\ K_{r,s}^{\Lambda,i+1}:=\set{z\in Z:\ \exists \bar z\in\Lambda,\ t_1\leq 0\leq t_2:\ z+t_{1,2}\bar z\in  K_{r,s}^{\Lambda,i}},\ i\geq 0.
\end{align*}
The lamination convex hull is then defined as $K_{r,s}^{\Lambda,lc}:=\cup_{i\geq 0}K_{r,s}^{\Lambda,i}$.
If then one obtains after a finite number of steps a set $K_{r,s}^{\Lambda,i}$ that is $\Lambda$-convex, then one has found the $\Lambda$-convex hull, since by definition one has $K_{r,s}^{\Lambda,i}\subset K_{r,s}^{\Lambda,lc}\subset K_{r,s}^{\Lambda}$, for all $i\geq 0$.

As already mentioned in the introduction, the lamination convex hull in general may be strictly smaller than the $\Lambda$-convex hull, see for instance the discussion in \cite{Hitruhin-Lindberg}. For classical examples of strict inclusion related to the theory of $T_4$-configurations, see for instance \cite{MSver,Sver,Sz-rank}.

Define $M_0:Z\to\R^{3\times 3}$ and $G:Z\to[0,+\infty)$ as
\begin{align}\label{eq:matrix}
M_0(z)=\alpha\otimes\beta-M+p\id,\quad G(z)=\sqrt{(r^2-|\alpha|^2)(s^2-|\beta|^2)}.
\end{align}
We have the following lower bound on the hull due to laminates.
\begin{proposition}\label{prop:lam}
Suppose that for $z\in Z$ there hold $|\alpha|\leq r$, $|\beta|\leq s$, and there exist $\bar\alpha,\bar\beta\in\mathbb R^3$ such that
\begin{align}
&M_0(z)+\bar\alpha\otimes\bar\beta=0,\label{eq:nnlam1}\\
&|\bar\alpha||\bar\beta|\leq G(z),\label{eq:nnlam1pt5}\\
&(\beta\cdot\bar\beta)\sqrt{r^2-|\alpha|^2}|\bar\alpha|=(\alpha\cdot\bar\alpha)\sqrt{s^2-|\beta|^2}|\bar\beta|,\label{eq:nnlam2}\\
&(\alpha-\beta)\cdot \bar\alpha\times\bar\beta=0\label{eq:nnlam4}.
\end{align}
Then $z\in K_{r,s}^{\Lambda,3}$.
\end{proposition}

We will briefly show that condition \eqref{eq:nnlam4} corresponds to (the relaxation of) Ohm's law.
To do so, for a skew-symmetric matrix $A\in\R^{3\times 3}$, define $f_0(A)\in\R^3$ as the unique vector such that $A\xi=\xi\times f_0(A)$, for all $\xi\in\R^{3}.$ Clearly, $f_0$ is linear.

\begin{remark}[The hidden relaxation of Ohm's law]\label{rem:ohm}
Let us make the important observation that when switching viewpoints from the Els\"asser one to the classical one, the function $\frac{1}{2}f_0(M-M^T)$ corresponds to the (relaxation of) the electric field $E$, as per  \cite{FLSz}. Indeed, for solutions of the ideal MHD equations \eqref{eq:MHD}, one can rewrite
$$\nabla\cdot(B\otimes u - u\otimes B)= \nabla\times(B\times u),$$
from where the correspondence follows, using the identities $E=B\times u$ and $B\otimes u - u\otimes B=\frac{1}{2}(M-M^T)$.

For subsolutions, the function $f_0(M-M^T)\cdot (\alpha-\beta)$ corresponds to $E\cdot B$ from \cite{FLSz} (using that $\alpha-\beta=2B$), and as shown in Section 3.2 from \cite{FLSz}, is $\Lambda$-affine and vanishes on $K_{r,s}$, as well as on $\Lambda$. Indeed, the latter follows by noting that for any $\bar z\in\Lambda$,
there exist $0\neq \xi\in\mathbb R^3$, $c\in\mathbb R$ such that
$$(\bar M -\bar M^T)\xi = - c(\bar\alpha-\bar\beta).$$
Setting $w:=f_0(\bar M-\bar M^T)$, we get
$$\xi\times w= - c(\bar\alpha-\bar\beta).$$
If $c\neq 0$, $\bar\alpha-\bar\beta$ is parallel to $\xi\times w$, thus $w\cdot (\bar\alpha-\bar\beta)=0$. If $c=0$, $w$ is parallel to $\xi$, and using the fact that $\bar z\in\Lambda$ also implies $(\bar\alpha-\bar\beta)\cdot \xi=0$, we once more have $w\cdot (\bar\alpha-\bar\beta)=0$. Thus, we have shown that the function $z\mapsto f_0(M-M^T)\cdot (\alpha-\beta)$ vanishes on $\Lambda$. Since it is quadratic, there follows that it is also $\Lambda$-affine. It is then straightforward to check that it also vanishes on $K_{r,s}$, thus putting the above together, it must also vanish on $K_{r,s}^\Lambda$, and as such, constitutes a hidden (relaxed) form of Ohm's law. 

Furthermore, in Section 6 of \cite{FLSz} it was also proved that the relative interior of $K_{r,s}^{\Lambda,lc}$ with respect to the manifold $\mathscr M=\set{z\in Z:\ f_0(M-M^T)\cdot(\alpha-\beta)=0}$ is non-empty. 
\end{remark}

\begin{remark}
On one hand, if 
\begin{align*}
A=\begin{pmatrix}
0 & a_3 & -a_2\\
-a_3 & 0 & a_1\\
a_2 & -a_1 & 0
\end{pmatrix},
\end{align*}
then one obtains $f_0(A)=(a_1,a_2,a_3)$.
On the other hand, using the Moore-Penrose pseudoinverse $A^\dagger$, a vector $\zeta\in\R^3$ is perpendicular to $f_0(A)$ if and only if
$$(\id-A^\dagger A)\zeta=0.$$
Hence, the condition $f_0(M-M^T)\cdot (\alpha-\beta)=0$ can be written equivalently as
\begin{align*}
(\|M-M^T\|^2_F\id+2(M-M^T)^2 )(\alpha-\beta)=0,
\end{align*}
where $\|\cdot\|_F$ denotes the Frobenius norm of a matrix.
\end{remark}

One can then easily get rid of the dependence on $\bar\alpha,\bar\beta$ from Proposition \ref{prop:lam} via the following equivalent conditions.
\begin{corollary}\label{cor:reform}
Suppose that for $z\in Z$ there hold $|\alpha|\leq r$, $|\beta|\leq s$, and
\begin{align}
&\text{rank}(M_0(z))\leq 1,\quad \|M_0(z)\|_n\leq G(z),\label{eq:mlam1}\\
&\alpha^T M_0(z)\beta\leq 0,\label{eq:mlam1pt5}\\
&|M_0(z)\beta|\sqrt{r^2-|\alpha|^2}=|M_0(z)^T\alpha|\sqrt{s^2-|\beta|^2},\label{eq:reformeq}\\
&(\alpha-\beta)\cdot f_0(M-M^T)=0.\label{eq:mlam4}
\end{align}
Then $z\in K_{r,s}^{\Lambda,3}$.
\end{corollary}
Above, $\|\cdot\|_n$ denotes the nuclear (or Ky Fan) norm of a matrix, i.e. the sum of its singular values.

\subsection{An upper estimate on the $\Lambda$-convex hull}

Recall \eqref{eq:matrix} and define
the relatively open set
\begin{multline}\label{eq:hull}
U=\left\{z\in Z:\ |\alpha|<r,\ |\beta|<s,\ f_0(M-M^T)\cdot(\alpha-\beta)=0,\ \|M_0(z)\|_n<G(z) \right\}.
\end{multline}

We have the following upper bound on the $\Lambda$-convex hull.

\begin{proposition}\label{prop:hull}
There holds $
K_{r,s}^\Lambda\subset\overline U.$
\end{proposition}

\begin{remark}[Comparison to the relaxation of the incompressible Euler equations]
Let us quickly compare our above result in the case when the magnetic field is additionally assumed to vanish with the relaxation obtained by De Lellis and Sz\'ekelyhidi Jr. for the incompressible Euler equations \cite{DeL-Sz-Annals}.
Of course for this to make sense, we need to formulate the two frameworks in a sufficiently equivalent manner, which requires some subtle care.

For our relaxation, it is straightforward to see due to the Els\"asser transformation that if $B=0$, one has $\alpha=\beta$ (thus we must also assume $r=s$), and implicitly $M=M^T$. Since these conditions are affine, and would hold on $K_{r,r}$, they should also be added to $U$ as an upper bound of $K^\Lambda_{r,r}.$

So our relaxation reduces to the linear equations
\begin{align}\label{eq:eq1}
\partial_t\alpha+\text{div } M =0,\quad \text{div }\alpha=0,
\end{align}
with now $M\in\mathcal S^{3\times 3}$, and
coupled with the point-wise constraints 
\begin{align}\label{eq:cons1}
|\alpha|\leq r,\quad \|\alpha\otimes\alpha-M+ p\id \|_n\leq r^2 - |\alpha|^2.
\end{align} 

Now let us carefully reformulate the relaxation of \cite{DeL-Sz-Annals} to a comparable form, recalling that a priori they hid the prescribed energy in the pressure, to allow for a traceless matrix $\sigma$ to appear under the divergence in the relaxed momentum equation. More precisely, their relaxation can be written as
\begin{align}\label{eq:eq2}
\partial_t v +\text{div } \left(\sigma+\frac{r^2}{3}\id\right)+\nabla p =0,\quad \text{div }v=0,
\end{align}
with  $\sigma\in\mathcal S_0^{3\times 3}$, and
coupled with the point-wise constraints 
\begin{align}\label{eq:cons2}
|v|\leq r,\quad \lambda_{\max} (v\otimes v-\sigma)\leq \frac{r^2}{3}.
\end{align} 
However, this yields that the matrix $\frac{r^2}{3}\id-v\otimes v+\sigma$ is positive-semidefinite, and
 the fact that $\sigma$ is trace-free and $|v|\leq r$ imply that
 \begin{align}\label{eq:cons3}
 \left\|\frac{r^2}{3}\id-v\otimes v+\sigma\right\|_n =  \text{tr}\left(\frac{r^2}{3}\id-v\otimes v+\sigma \right)=r^2-|v|^2.
 \end{align}
 Thus, it follows that given $(v,\sigma)$ satisfying \eqref{eq:eq2},\eqref{eq:cons2}, setting $\alpha:=v$ and $M:=\sigma+\left(\frac{r^2}{3}+p\right)\id$, yields that $(\alpha,M)$ satisfy \eqref{eq:eq1},\eqref{eq:cons1}. In other words, the convex hull associated with the relaxation of the incompressible Euler equations is included in our upper bound, as one would expect.
 
The main reason for which this implication cannot be inverted is that the trace of $M$ should be constant  $r^2+3p$, which is supposed to be a prescribed quantity, that does not oscillate during a convex integration framework, such that the associated $\sigma$ is indeed trace-free. This does not hold in general over our $\overline U$.

On the other hand, we also do not know a priori that for $(\alpha,M)$ satisfying \eqref{eq:cons1}, the matrix $M-\alpha\otimes\alpha-p\id=-M_0(z)$, which corresponds to $\frac{r^2}{3}\id-v\otimes v+\sigma$ is positive semi-definite. However, if the above trace condition were assumed, the sign of $M_0(z)$ could be determined a posteriori. Indeed, one has for any $A\in \mathcal S^{3\times 3}$ that
$$|tr(A)|\leq\|A\|_n,$$
with equality if and only if $A$ is semidefinite. For $A=M_0(z)=\alpha\otimes\alpha-M+p\id$, with the assumption $tr(M)=r^2+3p$, one would get $tr(M_0(z))=|\alpha|^2-r^2$, and since $z\in \overline U$ here implies $\|M_0(z)\|_n\leq -tr(M_0(z))$, there would also follow that $M_0(z)\leq 0$.

In conclusion, the only reason for which our relaxation (assuming vanishing magnetic field) is in general strictly larger than the relaxation of the incompressible Euler equations, is that in the appropriate reformulation, $\sigma$ cannot be made trace-free.
\end{remark}

\subsection{Comparison of the two estimates}\label{sec:conc}

Let us compare the sets in Propositions \ref{prop:lam} and \ref{prop:hull}.

\begin{remark}[Difficulties in calculating higher order laminates]\label{rem:rem1}
The main issue in laminating further beyond $K_{r,s}^{\Lambda,3}$ is that $\bar\alpha,\bar\beta$ from \eqref{eq:nnlam1}-\eqref{eq:nnlam4} depend in a non-explicit (and possibly highly non-linear) manner on $z$, thus making it hard to quantify the effect of perturbing $z$ along a $\Lambda$-segment can have on these vectors.
In particular, the most difficult part seems to come from \eqref{eq:nnlam2}. If instead one tries to laminate the reformulated conditions from Corollary \ref{cor:reform}, then again the difficulty comes from \eqref{eq:reformeq}, which when perturbed along a $\Lambda$-segment, squared and expanded with respect to $t$ yields a polynomial of degree eight.
\end{remark}

\begin{remark}[Difficulties in improving the upper bound]
One can prove that a given set which is compact and $\Lambda$-convex must be the hull by the means of a Krein-Milman type argument (see e.g. Lemma 4.16. from \cite{Kirchheim}), if one can show that  all $\Lambda$-extremal points of the given set are in $K_{r,s}$.

For our $\overline U$ this would reduce to showing that for any $z\in\partial U\setminus K_{r,s}$ there exists $\bar z\in \Lambda$ such that $z+t\bar z \in \overline U$ for small values of $t$. However, in the case when $M_0(z)$ is rank-one and its nuclear norm is equal to $G(z)$, one seems to be unable to achieve such perturbations, unless $z\in K_{r,s}^{\Lambda,1}$. 
\end{remark}

Lastly, we can compare the above observations with Remark \ref{rem:rem2} from the next section. Indeed, on one hand if $z\in\partial U\setminus K_{r,s}$ is such that $\|M_0(z)\|_n=G(z)$, $\text{rank}(M_0(z))=1$, then one can easily show that this is equivalent to the existence of $\bar\alpha,\bar\beta\in\mathbb R^3$ such that conditions \eqref{eq:nnlam1}, \eqref{eq:nnlam4} and \eqref{eq:flam1} hold.
On the other hand, we have seen in Remark \ref{rem:rem2} that $z\in K_{r,s}^{\Lambda,1}$ is equivalent to \eqref{eq:nnlam1}, \eqref{eq:nnlam4}, \eqref{eq:flam1} and \eqref{eq:flam2}. The gap between the two is then only the condition \eqref{eq:flam2}, i.e. $\alpha\cdot\bar\alpha=\beta\cdot\bar\beta$.

Hence we conjecture that some further $\Lambda$-convex inequality should be added to $U$ in order to reduce it appropriately such that the case $z\not\in K_{r,s}^{\Lambda,1}$, $\|M_0(z)\|_n=G(z)$, $\text{rank}(M_0(z))=1$ can be excluded, which may be related to some appropriate lamination of the condition $\alpha\cdot\bar\alpha=\beta\cdot\bar\beta$ (or the equivalent conditions \eqref{eq:nnlam2} or \eqref{eq:reformeq}).

\section{Proof of the lower bound}

In this section we prove Proposition \ref{prop:lam}.
We begin with the following complete characterization of the set of first-order laminates.
\begin{lemma}\label{lem:lam1}
Suppose that for $z\in Z$ there hold $|\alpha|<r$ and $|\beta|<s$. Then $z\in K_{r,s}^{\Lambda,1}$ if and only if there exist $\bar\alpha,\bar\beta\in\mathbb R^3\setminus\{0\}$ such that
\begin{align}
&M_0(z)+G(z)\frac{\bar\alpha}{|\bar\alpha|}\otimes\frac{\bar\beta}{|\bar\beta|}=0,\label{eq:lam1}\\
&(\beta\cdot\bar\beta) \frac{r^2-|\alpha|^2}{s^2-|\beta|^2}=\alpha\cdot\bar\alpha
\quad\text{and}\quad 
\frac{r^2-|\alpha|^2}{s^2-|\beta|^2}=\frac{|\bar\alpha|^2}{|\bar\beta|^2},\label{eq:lam2}\\
&(\alpha-\beta)\cdot \bar\alpha\times\bar\beta=0\label{eq:lam3}.
\end{align}
\end{lemma}
\begin{proof}
\textbf{Step 1.}
Let us first assume that \eqref{eq:lam1}-\eqref{eq:lam3} hold and let us prove that there exists $\bar z\in\Lambda$ such that $z+t_{1,2}\bar z\in K_{r,s}$ for some $t_1<0<t_2$.

We will set $\bar z:=(\bar \alpha,\bar \beta,\bar M)$ with $\bar \alpha,\bar \beta$ given by \eqref{eq:lam1}-\eqref{eq:lam3} and
$$\bar M=\alpha\otimes \bar\beta+\bar\alpha\otimes\beta - \frac{2\alpha\cdot\bar\alpha}{|\bar\alpha|^2}\bar\alpha\otimes\bar\beta.$$
It is then easy to check that $\bar z\in\Lambda$. Indeed, if $\bar\alpha\times\bar\beta\neq 0$, then setting $\xi:=\bar\alpha\times\bar\beta$, one has
\begin{align*}
\bar M \xi=(\beta\cdot\xi)\bar\alpha,\quad \bar M^T\xi=(\alpha\cdot\xi)\bar\beta,
\end{align*}
hence using \eqref{eq:lam3} we may set $c:=-\alpha\cdot\xi=-\beta\cdot\xi$ in order to obtain that
\begin{align*}
\begin{pmatrix}
\bar{M} & \bar{\alpha}\\
\bar{M^T} & \bar{\beta}\\
\bar{\alpha}^T & 0\\
\bar{\beta}^T & 0
\end{pmatrix} \begin{pmatrix}
\xi \\ c
\end{pmatrix}=0.
\end{align*}
On the other hand, if $\bar\alpha\times\bar\beta=0$, then $\bar\alpha\parallel\bar\beta$, so choosing any $$0\neq\xi\in(\alpha-\beta)^\perp\cap\bar\alpha^\perp$$ and once more $c:=-\alpha\cdot\xi=-\beta\cdot\xi$ allows us to conclude in a similar manner that $\bar z\in\Lambda$.

There remains to show that the system of equations
\begin{align*}
M_0(z+t \bar z)=0,\quad |\alpha+t\bar\alpha|=r,\quad |\beta+t\bar\beta|=s
\end{align*}
has two roots $t_1<0<t_2$. Squaring the latter two, one obtains the quadratic equations
\begin{align*}
t^2|\bar\alpha|^2+2t \alpha\cdot\bar\alpha + |\alpha|^2-r^2=0,\quad t^2|\bar\beta|^2+2t \beta\cdot\bar\beta + |\beta|^2-s^2=0.
\end{align*}
However, \eqref{eq:lam2} tells us exactly that these expressions have the same roots, which also satisfy $t_1<0<t_2$ due to $|\alpha|^2-r^2<0$. Also note that
$$t_{1,2}^2+2t_{1,2} \frac{\alpha\cdot\bar\alpha}{|\bar\alpha|^2}=\frac{r^2-|\alpha|^2}{|\bar\alpha|^2}=\frac{G(z)}{|\bar\alpha||\bar\beta|},$$
due to \eqref{eq:lam2}.

Finally, one may expand
\begin{align*}
M_0(z+t_{1,2} \bar z)&=M_0(z)+t_{1,2}\left(\alpha\otimes \bar\beta+\bar\alpha\otimes\beta-\bar M \right)+t_{1,2}^2 \bar\alpha\otimes\bar\beta
\\&=M_0(z)+\left(t_{1,2}^2+2t_{1,2} \frac{\alpha\cdot\bar\alpha}{|\bar\alpha|^2}\right)\bar\alpha\otimes\bar\beta
\\&=M_0(z)+\frac{G(z)}{|\bar\alpha||\bar\beta|}\bar\alpha\otimes\bar\beta,
\end{align*}
which is zero due to \eqref{eq:lam1}. Hence $z+t_{1,2}\bar z\in K_{r,s}$.

\textbf{Step 2.} Let us now show that if $z\in K_{r,s}^{\Lambda,1}$, then there exist $\bar\alpha,\bar\beta\in\mathbb R^3\setminus\{0\}$ satisfying
\eqref{eq:lam1}-\eqref{eq:lam3}. From the assumption on $z$ there follows that there exist $z_1,z_2\in K_{r,s}$, $\bar z:=z_1-z_2\in\Lambda$ such that $z=z_2+t\bar z=z_1-(1-t)\bar z$ for some $t\in(0,1)$.

First observe that there holds $\bar z=z_1-z_2\in\Lambda$ if and only if \eqref{eq:lam3} holds. Indeed, either $\bar\alpha\parallel\bar\beta$, hence \eqref{eq:lam3} would hold trivially, or one must pick $\xi$ to be (parallel to) $\bar\alpha\times\bar\beta$. However, one has
\begin{align*}
\bar M=\alpha_1\otimes\beta_1-\alpha_2\otimes\beta_2=\alpha_1\otimes\bar\beta+\bar\alpha\otimes\beta_2
=\alpha\otimes\bar\beta+\bar\alpha\otimes\beta+(1-2t)\bar\alpha\otimes\bar\beta,
\end{align*}
from where as in the first step,
$$\bar M \xi=(\beta\cdot\xi)\bar\alpha,\quad \bar M^T\xi=(\alpha\cdot\xi)\bar\beta,$$
thus $\bar z\in\Lambda$ is equivalent to $\alpha\cdot\xi=\beta\cdot\xi$, which is precisely \eqref{eq:lam3}.

Furthermore, since $z_2=z-t\bar z$ and $z_1=z+(1-t)\bar z$ are elements of $K_{r,s}$, we get 
\begin{align*}
|\alpha-t\bar\alpha|^2=r^2,\ |\beta-t\bar\beta|^2=s^2,\ |\alpha+(1-t)\bar\alpha|^2=r^2,\ |\beta+(1-t)\bar\beta|^2=s^2,
\end{align*}
which is possible exactly only if \eqref{eq:lam2} and $$t=\frac{1}{2}+\frac{\alpha\cdot\bar\alpha}{|\bar\alpha|^2}=\frac{1}{2}+\frac{\beta\cdot\bar\beta}{|\bar\beta|^2}$$
hold.
Thus, in particular there also holds
$$t(1-t)=\frac{G(z)}{|\bar\alpha||\bar\beta|}.$$
Finally, it is an easy calculation to check that, under these conditions, the fact that $M_0(z_i)=0$ holds yields precisely \eqref{eq:lam1}.
\end{proof}

Let us now state a result which is essentially the same as Lemma 6.10 from \cite{FLSz}, however we provide a slightly different proof.
\begin{lemma}\label{lem:lam2}
For any $z\in Z$ with $|\alpha|\leq r$, $|\beta|\leq s$ and $M_0(z)=0$, there holds $z\in K_{r,s}^{\Lambda,2}$.
\end{lemma}
\begin{proof} 
\textbf{Step 1.} Let us show that 
$$\set{z\in Z:\ |\alpha|\leq r,\ |\beta|=s,\ M_0(z)=0}\subset K_{r,s}^{\Lambda,1}.$$
If $|\alpha|=r$, we are in $K_{r,s}$, so the result follows trivially.

If $|\alpha|<r$, then if $\alpha\neq 0$, set $\bar\alpha:=\frac{\alpha}{|\alpha|}$, otherwise pick it to be any arbitrary unit vector, and set $\bar\beta:=0$, $\bar M:=\bar\alpha\otimes\beta$. Then, choosing any $0\neq\xi\in\bar\alpha^\perp\cap\beta^\perp$ and $c:=0$ easily yields $\bar z\in\Lambda$.

We also trivially have that $M_0(z+t\bar z)=M_0(z)$, for all $t\in\mathbb R$, and clearly $|\alpha+t\bar\alpha|^2=r^2$ has two roots $t_1<0<t_2$. Hence $z+t_{1,2}\bar z\in K_{r,s}$

\textbf{Step 2.} Now let us prove that 
$$\set{z\in Z:\ |\alpha|\leq r,\ |\beta|\leq s,\ M_0(z)=0}\subset K_{r,s}^{\Lambda,2}.$$
We aim to reach the set from Step 1 by moving along appropriate $\Lambda$-segments.
We may suppose that $|\beta|< s$, otherwise the result would follow from Step 1.
If $\beta\neq 0$, set $\bar\beta:=\frac{\beta}{|\beta|}$, otherwise pick it to be any arbitrary unit vector, and set now $\bar\alpha:=0$, $\bar M:=\alpha\otimes\bar\beta$. Once more, choosing any $0\neq\xi\in\alpha^\perp\cap\bar\beta^\perp$ and $c:=0$ easily yields $\bar z\in\Lambda$.

As before, we have that $M_0(z+t\bar z)=M_0(z)$, for all $t\in\mathbb R$, and $|\beta+t\bar\beta|^2=s^2$ has two roots $t_1<0<t_2$. Hence, using Step 1 we may deduce that $z+t_{1,2}\bar z\in K_{r,s}^{\Lambda,1}$, which concludes the proof of the lemma.
\end{proof}

As a consequence, we have the following result.
\begin{lemma}\label{lem:lam3}
Suppose that for $z\in Z$ there hold  $|\alpha|<r$, $|\beta|<s$, and there exist $c\in(0,G(z)]$ and $\bar\alpha,\bar\beta\in\mathbb R^3\setminus\{0\}$ such that
\begin{align}
&M_0(z)+c\frac{\bar\alpha}{|\bar\alpha|}\otimes\frac{\bar\beta}{|\bar\beta|}=0,\label{eq:nlam1}\\
&(\beta\cdot\bar\beta) \frac{r^2-|\alpha|^2}{s^2-|\beta|^2}=\alpha\cdot\bar\alpha
\quad\text{and}\quad 
\frac{r^2-|\alpha|^2}{s^2-|\beta|^2}=\frac{|\bar\alpha|^2}{|\bar\beta|^2},\label{eq:nlam2}\\
&(\alpha-\beta)\cdot \bar\alpha\times\bar\beta=0\label{eq:nlam3}.
\end{align}
Then $z\in K_{r,s}^{\Lambda,3}$.
\end{lemma}
\begin{proof}
We proceed as in the corresponding step of the proof of Lemma \ref{lem:lam1}, considering 
 $\bar z:=(\bar \alpha,\bar \beta,\bar M)$ with $\bar \alpha,\bar \beta$ given by \eqref{eq:nlam1}-\eqref{eq:nlam3} and
$$\bar M=\alpha\otimes \bar\beta+\bar\alpha\otimes\beta - \frac{2\alpha\cdot\bar\alpha}{|\bar\alpha|^2}\bar\alpha\otimes\bar\beta.$$
Proving that $\bar z\in\Lambda$ works precisely the same way as in the proof of Lemma \ref{lem:lam1}.

The main difference now is that we no longer need to reach $|\alpha+t\bar\alpha|=r$ and $|\beta+t\bar\beta|=s$. Indeed, it suffices to show that there exist $t_1<0<t_2$ such that $M_0(z+t_{1,2}\bar z)=0$, $|\alpha+t_{1,2}\bar\alpha|\leq r$ and $|\beta+t_{1,2}\bar\beta|\leq s$, since this would imply that $z+t_{1,2}\bar z\in  K_{r,s}^{\Lambda,2}$ by virtue of Lemma \ref{lem:lam2}.

We obtain by \eqref{eq:nlam1} that
\begin{align*}
M_0(z+t\bar z)=\left(t^2+2t\frac{\alpha\cdot\bar\alpha}{|\bar\alpha|^2}-\frac{c}{|\bar\alpha||\bar\beta|}\right)\bar\alpha\otimes\bar\beta,
\end{align*}
and since $c>0$, the quadratic equation in the parentheses has two roots $t_1<0<t_2$. Furthermore, from \eqref{eq:nlam2} we have $$\frac{r^2-|\alpha|^2}{|\bar\alpha|^2}=\frac{G(z)}{|\bar\alpha||\bar\beta|},$$
which we may combine with $c\leq G(z)$ in order to obtain that $|\alpha+t_{1,2}\bar\alpha|\leq r$, and similarly that $|\beta+t_{1,2}\bar\beta|\leq s$. This concludes the proof of the lemma.
\end{proof}

Finally, we can combine all our above results in order to prove our main result for this section, which also encompasses the cases with non-strict inequalities.

\begin{proof}[Proof of Proposition \ref{prop:lam}]

The proof follows by appropriately combining Lemmas \ref{lem:lam2} and  \ref{lem:lam3}. Assume that the conditions of Proposition \ref{prop:lam} hold.

If either $G(z)$, $|\bar \alpha|$ or $|\bar \beta|$ is zero then we have $M_0(z) = 0$, hence, we are in the case of Lemma \ref{lem:lam2}, and the assertion of the proposition follows.

If all three of the above mentioned quantities are non-zero, we show that the conditions of Lemma \ref{lem:lam3} also hold. We easily obtain the validity of condition \eqref{eq:nlam1} for $\bar \alpha$ and $\bar\beta$ of Proposition \ref{prop:lam} by combining \eqref{eq:nnlam1} and \eqref{eq:nnlam1pt5}, and setting $c=|\bar \alpha||\bar\beta|$. Condition \eqref{eq:nlam1} remains valid for the rescaled vectors $\widetilde\alpha=\lambda\bar\alpha$ and $\widetilde\beta=\frac{1}{\lambda}\bar\beta$ with any $\lambda\neq 0$. Let us choose $\lambda$ such that the first condition from \eqref{eq:nlam2} holds for $\widetilde\alpha$ and $\widetilde\beta$. It is easy to see that the choice
$$
\lambda
=
\left(
\sqrt{\frac{r^2-|\alpha|^2}{s^2-|\beta|^2}}
\frac{|\bar\beta|}{|\bar \alpha|}
\right)^{\frac{1}{2}}
$$
is appropriate, and it also implies the second condition from \eqref{eq:nlam2}. Finally, due to \eqref{eq:nnlam4}, condition \eqref{eq:nlam3} holds trivially for $\widetilde\alpha$ and $\widetilde\beta$.
\end{proof}

\begin{remark}[A reformulation of the first-order laminates]\label{rem:rem2}
If, in addition, one effectuates the above transformations in the case of Lemma \ref{lem:lam1}, then one may also obtain the following equivalent characterization: if  $|\alpha|<r$ and $|\beta|<s$, then there holds $z\in K_{r,s}^{\Lambda,1}$ if and only if there exist $\bar\alpha,\bar\beta\in\mathbb R^3$ such that
there hold \eqref{eq:nnlam1}, \eqref{eq:nnlam4}, as well as 
\begin{align}\label{eq:flam1}
|\bar\alpha|^2=s^2-|\beta|^2,\quad|\bar\beta|^2=r^2-|\alpha|^2
\end{align}
and
\begin{align}\label{eq:flam2}
\alpha\cdot\bar\alpha=\beta\cdot\bar\beta.
\end{align}
\end{remark}

We conclude this section with the proof of Corollary \ref{cor:reform}.
\begin{proof}[Proof of Corollary \ref{cor:reform}]
We show that  the conditions of Proposition \ref{prop:lam} follow from conditions \eqref{eq:mlam1}-\eqref{eq:mlam4}.

If $M_0(z)=0$, then the conditions of the proposition  clearly hold by choosing $\bar\alpha=\bar\beta=0$.   
If $M_0(z)\neq 0$, then due to \eqref{eq:mlam1} there exist $\bar\alpha\neq0$ and $\bar\beta\neq0$ such that
$$
M_0(z)=-\bar\alpha\otimes\bar\beta
\qquad\mbox{and}\qquad
\|M_0(z)\|_n=|\bar\alpha||\bar\beta|\leq G(z),
$$
thus \eqref{eq:nnlam1} and \eqref{eq:nnlam1pt5} hold. From \eqref{eq:mlam1pt5} we obtain 
\begin{equation}
\label{eq:samesign}
0\geq\alpha^T M_0(z)\beta
=
-
(\alpha\cdot\bar\alpha)(\bar\beta\cdot\beta),
\end{equation}
thus $\alpha\cdot\bar\alpha$ and $\bar\beta\cdot\beta$ have the same sign. Using \eqref{eq:reformeq} we get
\begin{multline*}
|(\bar\beta\cdot\beta)\bar\alpha|\sqrt{r^2-|\alpha|^2}
=
|M_0(z)\beta|\sqrt{r^2-|\alpha|^2}
=\\
|M_0(z)^T\alpha|\sqrt{s^2-|\beta|^2}
=
|(\alpha\cdot\bar\alpha)\bar\beta|\sqrt{s^2-|\beta|^2},
\end{multline*}
which, together with \eqref{eq:samesign},  gives condition \eqref{eq:nnlam2}.

Finally, in addition to the abstract arguments from Remark \ref{rem:ohm}, let us explain how condition \eqref{eq:nnlam4} can be obtained from condition \eqref{eq:mlam4} directly.

A simple calculation yields
\begin{align*}
M-M^T=\alpha\otimes\beta-\beta\otimes\alpha+\bar\alpha\otimes\bar\beta-\bar\beta\otimes\bar\alpha,
\end{align*}
and using the identity
$$f_0(a\otimes b - b\otimes a)=a\times b,$$
this gives us
$$f_0(M-M^T)=\alpha\times\beta+\bar\alpha\times\bar\beta.$$
From \eqref{eq:mlam4} we have
$$0=(\alpha\times\beta+\bar\alpha\times\bar\beta)\cdot(\alpha-\beta)=\bar\alpha\times\bar\beta\cdot(\alpha-\beta),$$
which is precisely \eqref{eq:nnlam4}.

Thus, the conditions of Proposition \ref{prop:lam} are fulfilled and the assertion of the Corollary follows.
\end{proof}

\section{Proof of the upper bound}\label{sec:convex_hull_unconstraint}

In this section we prove Proposition \ref{prop:hull}.

\begin{proof}[Proof of Proposition \ref{prop:hull}]

Since we have that, for any $A\in\mathbb R^{3\times3}$, the nuclear norm can be written as
$$\|A\|_n=\sup_{\|Q\|_2\leq 1}\text{tr}(Q^T A),$$
let us show that, for any $\gamma\in[0,1]$, $Q\in\mathbb R^{3\times3}$ with $\|Q\|_2\leq 1$, the function
$$H_{\gamma,Q}(z):=\gamma(|\alpha|^2-r^2)+(1-\gamma)(|\beta|^2-s^2)+2\sqrt{\gamma(1-\gamma)}\text{tr}(Q^T M_0(z))$$
is convex on $Z$.

We take $\bar z\in Z$, $t\in \R$ and estimate
\begin{multline*}
\text{tr}(Q^T M_0(z+t\bar z))-\text{tr}(Q^T M_0(z))
= t \text{tr}(Q^T (\alpha\otimes\bar\beta+\bar\alpha\otimes\beta-\bar M))+
t^2 \bar\alpha\cdot Q\bar\beta.
\end{multline*}
Hence we have
that $H_{\gamma,Q}(z+t \bar z)$ is a quadratic polynomial with $t^2$-coefficient given by
\begin{multline*}
\gamma|\bar\alpha|^2+(1-\gamma)|\bar\beta|^2+2\sqrt{\gamma(1-\gamma)}\bar\alpha\cdot Q\bar\beta \\ \geq \gamma|\bar\alpha|^2+(1-\gamma)|\bar\beta|^2-2\sqrt{\gamma(1-\gamma)}|\bar\alpha| |\bar\beta|\geq 0,
\end{multline*}
from where the convexity of $H_{\gamma,Q}$ follows.

We may then conclude that for any $\gamma\in[0,1]$, the function 
$$H_\gamma(z):=\gamma(|\alpha|^2-r^2)+(1-\gamma)(|\beta|^2-s^2)+2\sqrt{\gamma(1-\gamma)}\|M_0(z)\|_n$$
is also convex on $Z$, by taking the supremum over $Q$ of a family of convex functions.

Since $H_\gamma=0$ on $K_{r,s}$, we get that $H_\gamma\leq 0$ on $K^\Lambda_{r,s}$, for all $\gamma\in[0,1]$. Taking $\gamma=0$ and $\gamma=1$ yield respectively $|\beta|\leq s$ and $|\alpha|\leq r$. If $\gamma\in(0,1)$, one may rewrite $H_\gamma\leq 0$ as
$$\|M_0(z)\|_n\leq \frac{\gamma(r^2-|\alpha|^2)+(1-\gamma)(s^2-|\beta|^2)}{2\sqrt{\gamma(1-\gamma)}},$$
thus we have shown that 
\begin{multline*}
\bigcap_{\gamma\in[0,1]}\set{H_\gamma\leq 0}\\=\bigcap_{\gamma\in(0,1)}\set{z\in Z:\ |\alpha|\leq r,\ |\beta|\leq s,\ \|M_0(z)\|_n\leq \frac{\gamma(r^2-|\alpha|^2)+(1-\gamma)(s^2-|\beta|^2)}{2\sqrt{\gamma(1-\gamma)}}}
\end{multline*}
is a convex set containing $K^\Lambda_{r,s}$. One may then, for each $z\in Z$ with $ |\alpha|\leq r,\ |\beta|\leq s$, minimize  the expression on the right-hand side of the last inequality in the sets above to get that
$$G(z)=\inf_{\gamma\in(0,1)}\frac{\gamma(r^2-|\alpha|^2)+(1-\gamma)(s^2-|\beta|^2)}{2\sqrt{\gamma(1-\gamma)}},$$
 where the infimum is achieved when both deficits are positive, and obtained by a boundary limit otherwise, for 
$$\bar\gamma:=\frac{s^2-|\beta|^2}{r^2-|\alpha|^2+s^2-|\beta|^2}\in[0,1].$$

Therefore, it follows that
$$\bigcap_{\gamma\in[0,1]}\set{H_\gamma\leq 0}=\set{ z\in Z:\ |\alpha|\leq r,\ |\beta|\leq s,\ \|M_0(z)\|_n\leq G(z)},$$
which as we have shown above, is convex. Finally, the proof of the Proposition follows by recalling that $f_0(M-M^T)\cdot(\alpha-\beta)=0$ is a $\Lambda$-affine condition, thus $\overline U$ is a $\Lambda$-convex set containing $K_{r,s}$, and consequently also $K^\Lambda_{r,s}$.
\end{proof}

\begin{remark}
We note here that although the above proof can be adapted to use any Schatten norm instead of the nuclear norm, our choice of $U$ is, in a sense, optimal due to the monotonicity of the Schatten norms, with the nuclear norm being the largest. Furthermore, the nuclear norm  plays a key role in the convex hull of rank-one matrices. Since in Lemma \ref{lem:lam1} we are essentially moving along a rank-one perturbation of $M_0(z)$, this further suggests that the natural choice is the nuclear norm. However, it is an open question regarding the possible reduction of the upper bound, whether similar convex functions to $H_\gamma$ above can be considered, namely
$$\tilde H_\gamma(z):=\gamma(|\alpha|^2-r^2)+(1-\gamma)(|\beta|^2-s^2)+2\sqrt{\gamma(1-\gamma)} g(z),$$
where $g$ possibly does not even depend on $M_0(z)$ at all.
 Of course for the arguments of the proof to work, one would need a estimate of the type
 $$g(z+t\bar z)-g(z)\geq c_1 t - t^2|\bar\alpha||\bar\beta|.$$
\end{remark}

\section*{Acknowledgements} 

The authors thank Sauli Lindberg for insightful conversations provided on the topic.

\vspace{30pt}
\noindent University of Debrecen, Faculty of Science and Technology, Institute of Mathematics, 4032 Debrecen Egyetem t\'er 1, Hungary\\
\texttt{borbala.fazekas@science.unideb.hu}\\
\noindent Department of Analysis and Operations Research, Institute of Mathematics, Budapest University of Technology and Economics, Műegyetem rkp. 3., H-1111 Budapest, Hungary,  and
HUN-REN Alfr\'ed R\'enyi Institute of Mathematics, 1053 Budapest, Re\'altanoda utca 13-15, Hungary,
jkolumban@math.bme.hu, Corresponding author \\
\texttt{jkolumban@math.bme.hu}


\begin{thebibliography}{99}

\bibitem{Alf} H. Alfvén,  Existence of Electromagnetic-Hydrodynamic Waves, Nature (1942), 150 (3805), 405–406.



\bibitem{Castro_Cordoba_Faraco} \'A. Castro, D. C\'ordoba, D. Faraco, Mixing solutions for the Muskat problem, Invent. Math. 226.1 (2021), 251--348. 

\bibitem{Castro_Faraco_Mengual} \'A. Castro, D. Faraco, F. Mengual, Degraded mixing solutions for the Muskat problem, Calc. Var. Partial Differential Equations 58.2 (2019).

\bibitem{Castro_Faraco_Mengual_2} \'A. Castro, D. Faraco, F. Mengual, Localized mixing zone for Muskat bubbles and turned interfaces,  Ann. PDE 8, 7 (2022). https://doi.org/10.1007/s40818-022-00121-w



\bibitem{Crippa} G. Crippa, N. Gusev, S. Spirito, E. Wiedemann,  Non-Uniqueness and prescribed energy for the continuity equation, Comm. in Math. Sciences 13.7 (2015), 1937--1947. 



\bibitem{DeL-Sz-Annals} C. De Lellis, L. Sz\'ekelyhidi Jr., The Euler equations as a differential inclusion, Ann. Math. 170.3 (2009), 1417--1436.

\bibitem{DeL-Sz-Adm} C. De Lellis, L. Sz\'ekelyhidi Jr., On admissibility criteria for weak solutions of the Euler equations, Arch. Rat. Mech. Anal. 195.1 (2010), 225--260.



\bibitem{v2} A. M. Dykhne, E. P. Velikhov, Plasma turbulence due to the ionization instability in a strong magnetic field, Proceedings, 6th International Conference on Phenomena in Ionized Gases, Vol. 4. Paris, France. (1963) p. 511.

\bibitem{v3} A. M. Dykhne, I. Y. Shipuk,  E. P. Velikhov, Ionization instability of a plasma with hot electrons, 7th International Conference on Ionization Phenomena in Gases, Belgrade, Yugoslavia (1965).

\bibitem{FLSz} 
 D. Faraco, S. Lindberg, and L. Sz\'ekelyhidi, Jr., Bounded solutions of ideal MHD with compact
support in space-time, Arch. Ration. Mech. Anal. 239 (2021), no. 1, 51–93.
 


\bibitem{FLSz2} D. Faraco, S. Lindberg, and L. Sz\'ekelyhidi, Jr., Rigorous results on conserved and dissipated quantities in ideal MHD turbulence,
Geophys. Astrophys. Fluid Dyn. 116 (2022), no. 4, 237–260.

\bibitem{FLSz3} D. Faraco, S. Lindberg, and L. Sz\'ekelyhidi, Jr.,  Magnetic helicity, weak solutions and relaxation of ideal MHD, (2024) Comm. Pure Appl. Math., 77, 2387-2412.



\bibitem{Foerster-Sz} C. F\"orster, L. Sz\'ekelyhidi Jr., Piecewise constant subsolutions for the Muskat problem, Commun. Math. Phys. 363.3 (2018), 1051--1080.



\bibitem{GKSz} B. Gebhard, J. J. Kolumb\'an, L. Sz\'ekelyhidi Jr., A new approach to the Rayleigh-Taylor
instability, Arch. Rat. Mech. Anal. 241 (2021), 1243--1280.


\bibitem{GK} B. Gebhard, J. J. Kolumb\'an, Relaxation of the Boussinesq system and applications to the Rayleigh-Taylor instability, Nonlinear Differential Equations and Applications 29, 7 (2022). DOI: 10.1007/s00030-021-00739-y 

\bibitem{GKEE} B. Gebhard and J. J. Kolumb\'an. On bounded two-dimensional globally dissipative
Euler flows. SIAM J. Math. Anal., 54(3):3457-3479, 2022.

\bibitem{GKRTE} B. Gebhard, J. J. Kolumb\'an, The Rayleigh-Taylor instability with local energy dissipation, Math. Ann. 393, 3283–3336 (2025).

\bibitem{GBL} J.-F. Gerbeau, C. Le Bris, T. Leli\`evre, Mathematical methods for the magnetohydrodynamics of liquid metals, Numerical Mathematics and Scientific Computation, Oxford
University Press, Oxford, (2006).


\bibitem{Hitruhin-Lindberg} L. Hitruhin, S. Lindberg, The lamination convex hull of stationary incompressible porous media equations, SIAM J. Math. Anal. 53.1 (2021), 491--508.

\bibitem{dyn} L. Hitruhin, S. Lindberg, Relaxation of the kinematic dynamo equations, Proc. Amer. Math. Soc. 152 (2024), 5265-5278.



\bibitem{Kirchheim} B. Kirchheim, Rigidity and Geometry of microstructures, Habilitation thesis, University of Leipzig (2003).

\bibitem{Lax} P. D. Lax, The zero dispersion limit, a deterministic analogue of turbulence, Comm. Pure
Appl. Math. 44 (1991), no. 8-9, 1047–1056.




\bibitem{Mengual} F. Mengual, H-principle for the 2D incompressible porous media equation with viscosity jump, Analysis and PDE 15 (2022) 429–476.

\bibitem{Mengual_Sz_vortex_sheet} F. Mengual, L. Sz\'ekelyhidi Jr., Dissipative Euler flows for vortex sheet initial data without distinguished sign, Comm. Pure Appl. Math., 76, 163-221. https://doi.org/10.1002/cpa.22038

\bibitem{MSver} S. Müller and V. Šverák,
Convex integration with constraints and applications to phase transitions and partial differential equations,
J. Eur. Math. Soc. 1 (1999), 393–422.

\bibitem{Noisette-Sz} F. Noisette, L. Sz\'ekelyhidi Jr., Mixing solutions for the Muskat problem with variable speed,  J. Evol. Equ. 21, 3289–3312 (2021). https://doi.org/10.1007/s00028-020-00655-1



\bibitem{ST} M. Sermange, R. Temam, Some mathematical questions related to the MHD equations.
Comm. Pure Appl. Math. 36, no. 5, 635–664 (1983).

\bibitem{Sver} V. Šverák,
Rank-one convexity does not imply quasiconvexity,
Proc. Roy. Soc. Edinburgh Sect. A 120 (1992), 185–189.

\bibitem{Sz-rank} L. Székelyhidi Jr.,
Rank-one convex hulls in $\mathbb R^{2\times 2}$,
Calc. Var. Partial Differential Equations 22 (2005), 253–281.

\bibitem{Sz-Muskat} L. Sz\'ekelyhidi Jr., Relaxation of the incompressible porous media equation, Ann. Scient. \'Ec. Norm. Sup. 45.3 (2012), 491--509.

\bibitem{Sz-KH} L. Sz\'ekelyhidi Jr., Weak solutions to the incompressible Euler equations with vortex sheet initial data, C. R. Acad. Sci. Paris, Ser. I 349 (2011), 1063--1066.



\bibitem{Tartar} L. Tartar, The compensated compactness method applied to systems of conservation laws, NATO Adv. Sci. Inst. Ser. C Math. Phys. Sci. 111, Reidel, Dordrecht (1983), 263--285.

\bibitem{v1} E. P. Velikhov, Paper 47, Hall instability of current carrying slightly ionized plasmas, 1st International Conference on MHD Electrical Power Generation, Newcastle upon Tyne, England. (1962) p. 135.







\end{thebibliography}
\end{document}